\documentclass[12pt]{article} 
\usepackage{amssymb, amsmath, amsthm}
\usepackage{geometry} 
\usepackage{amssymb}
\usepackage{amsfonts} 
\usepackage{amsmath} 
\usepackage{color}
\usepackage{cancel}
\usepackage{mathrsfs}
\usepackage{yfonts}
\usepackage{accents}
\usepackage{commath}
\usepackage{relsize}
\usepackage{graphicx}
\usepackage[utf8]{inputenc}
\usepackage{cite}
\usepackage{multicol}
\usepackage{tabularx}
\usepackage{multirow}
\usepackage{float}
\usepackage{relsize}
\usepackage{array}
\usepackage{adjustbox}
\usepackage{comment}
\usepackage{color}

\newcommand{\R}{\mathbb R}
\newcommand{\rn}{\R^N}

\newcommand{\vol}{{\rm Vol}}
\newcommand{\per}{{\rm Per}}

\newtheorem{thm}{Theorem}[section]
\newtheorem*{thm*}{Theorem}
\newtheorem{lem}[thm]{Lemma}
\newtheorem{rmk}[thm]{Remark}
\newtheorem{prop}[thm]{Proposition}

\usepackage{tikz}
\newcommand*\circled[1]{\tikz[baseline=(char.base)]{
\node[shape=circle,draw,inner sep=2pt] (char) {#1};}}

\newcommand{\ct}[1]{\langle {#1}\rangle \lower.3ex\hbox{$_{t}$}}
\newcommand{\lt}[1]{[ {#1}] \lower.3ex\hbox{$_{t}$}}

\newcommand{\id}{{{\rm Id}}}

\newcommand{\C}{{\mathcal C}}

\newcommand{\A}{{\mathcal{A}}}
\newcommand{\B}{{\mathcal{B}}}

\newcommand{\0}{{\bf 0}}
\newcommand\restr[2]{{
\left.\kern-\nulldelimiterspace 
#1 
\vphantom{|} 
\right|_{#2} 
}}
\newcommand\jump[1]{\left[#1 \right]}
\newcommand{\ldue}{{l_2^E(B_R)(\hn\hn)}}

\newcommand{\overbar}[1]{\mkern 1.5mu\overline{\mkern-1.5mu#1\mkern-1.5mu}\mkern 1.5mu}
\usepackage{bbm}
\newcommand{\dv}{{\rm{div}}}
\newcommand{\h}{{{\bf {h}}}}
\newcommand{\n}{{{\bf {n}}}}

\newcommand{\compt}{\Omega\setminus \overbar{\omega_t}}
\newcommand{\iO}{\int_\Omega}

\newcommand{\dn}{\partial_n}

\newcommand{\gr}{\nabla}
\newcommand{\lap}{\Delta}

\newcommand{\hn}{h_n}
\newcommand{\sg}{\sigma}

\newcommand{\intf}{{\partial \omega}}
\newcommand{\br}{{B_R}}
\newcommand{\bbr}{{\partial B_R}}
\newcommand{\compl}{{\Omega\setminus \overline{B_R}}}
\newcommand{\isl}{\int_{\bbr}}
\newcommand{\hh}{h_n^2}
\newcommand{\htau}{{\h_\tau}}

\title{Locally optimal configurations for the two-phase torsion problem in the ball}

\author{Lorenzo Cavallina \thanks{This research was partially supported by the Grant-in-Aid for Scientific Research (B) (\#26287020) Japan Society for the Promotion of Science}}
\date{}
\begin{document}
\maketitle

\begin{abstract}
We consider the unit ball $\Omega\subset \rn$ ($N\ge2$) filled with two materials with different conductivities. We perform shape derivatives up to the second order to find out precise information about locally optimal configurations with respect to the torsional rigidity functional. 
In particular we analyse the role played by the configuration obtained by putting a smaller concentric ball inside $\Omega$.
In this case the stress function admits an explicit form which is radially symmetric: this allows us to compute the sign of the second order shape derivative of the torsional rigidity functional with the aid of spherical harmonics.  
Depending on the ratio of the conductivities a symmetry breaking phenomenon occurs. 
\end{abstract}
\bigskip


\bigskip


\tableofcontents

\section{Introduction}
We will start by considering the following two-phase problem.
Let $\Omega\subset \R^N$ ($N\ge 2$) be the unit open ball centered at the origin.
Fix $m\in (0,	\vol(\Omega))$, where here we denote the $N$-dimensional Lebesgue measure of a set by $\vol (\cdot)$ . 
Let $\omega\subset\subset \Omega$ be a sufficiently regular open set such that $\vol(\omega)=m$. 
Fix two positive constants $\sigma_-$, $\sigma_+$ and consider the following {\it distribution of conductivities}: $$\sigma:=\sigma_\omega:= \sigma_- \mathbbm{1}_\omega+\sigma_+ \mathbbm{1}_{\Omega\setminus \overbar{\omega}},$$
where by $\mathbbm{1}_A$ we denote the characteristic function of a set $A$ (i.e.\ $\mathbbm{1}_A(x)=1$ if $x\in A$ and vanishes otherwise).
\begin{figure}[h]
\centering
\includegraphics[scale=0.5]{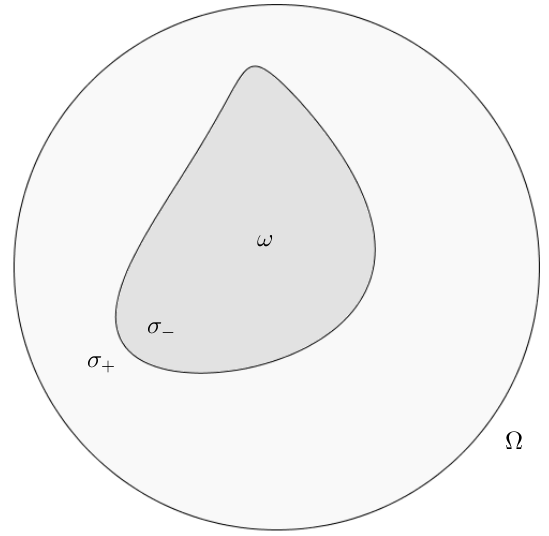}
\caption{Our problem setting}
\end{figure}

Consider the following boundary value problem:
\begin{equation}\label{pb}
\begin{cases}
-\dv{\left(\sigma_\omega \nabla u \right)}= 1 \quad&\text{ in } \Omega, \\
u=0 \quad&\text{ on }\partial \Omega.
\end{cases}
\end{equation}

We recall the weak formulation of \eqref{pb}: 
\begin{equation}\label{weak}
\int_\Omega \sigma_\omega \nabla u \cdot \nabla \varphi = \int_\Omega \varphi \;\;\;\;\text{ for all } \varphi\in H^1_0(\Omega).
\end{equation}

Moreover, since $\sigma_\omega$ is piecewise constant, we can rewrite \eqref{pb} as follows
\begin{equation}\label{pb2}
\begin{cases}
-\sigma_\omega \Delta u = 1 \quad&\text{ in } \omega\cup \left(\Omega\setminus\overbar{\omega}\right), \\
\sigma_-\partial_n u_-= \sigma_+\partial_n u_+ \quad&\text{ on } \partial \omega,\\
u=0 \quad&\text{ on } \partial \Omega,
\end{cases} 
\end{equation}
where the following notation is used: the symbol $\n$ is reserved for the outward unit normal and $\dn:=\frac{\partial}{\partial \n}$ denotes the usual normal derivative. Throughout the paper we will use the $+$ and $-$ subscripts to denote quantities in the two different phases (under this convention we have $(\sg_\omega)_+=\sg_+$ and $(\sg_\omega)_-=\sg_-$ and our notations are ``consistent" at least in this respect). The second equality of \eqref{pb2} has to be intended in the sense of traces. In the sequel, the notation $[f]:=f_+-f_-$ will be used to denote the jump of a function $f$ through the interface $\partial \omega$ (for example, following this convention, the second equality in \eqref{pb2} reads ``$\jump{\sg \dn u}=0$ on $\partial \omega$'').  

We consider the following {\em torsional rigidity functional}: 
$$
E(\omega):= \int_\Omega \sigma_\omega |\nabla u_\omega|^2=\int_\omega \sg_- |\gr u_\omega|^2+\int_{\Omega\setminus \bar{\omega}} \sg_+ |\gr u_\omega|^2=
\int_\Omega u_\omega, 
$$
where $u_\omega$ is the unique (weak) solution of \eqref{pb}.

Physically speaking, we imagine our ball $\Omega$ being filled up with two different materials and the constants $\sg_{\pm}$ represent how ``hard" they are. The second equality of \eqref{pb2}, which can be obtained by integrating by parts after splitting the integrals in \eqref{weak}, is usually referred to as {\em transmission condition} in the literature and can be interpreted as the continuity of the flux through the interface $\intf$.\\
The functional $E$, then, represents the torsional rigidity of an infinitely long composite beam. Our aim is to study (locally) optimal shapes of $\omega$ with respect to the functional $E$ under the fixed volume constraint.
The one-phase version of this problem was first studied by P\'olya. 
Let $D\subset \rn$ ($N\ge 2$) be a bounded Lipschitz domain. 
Define the following shape functional 
$$
\mathcal{E}(D):= \int_D |\gr u_D|^2, 
$$ 
where the function $u_D$ (usually called \emph{stress function}) is the unique solution to 
\begin{equation}
\begin{cases}
-\lap u = 1 \quad &\text{in } D,\\
u=0\quad &\text{on } \partial D.
\end{cases}
\end{equation}

The value $\mathcal{E}(D)$ represents the torsional rigidity of an infinitely long beam whose cross section is given by $D$. 
The following theorem (see \cite{polya}) tells us that beams with a spherical section are the ``most resistant''. 

\begin{thm}[P\'olya]\label{pol}
The ball maximizes $\mathcal{E}$ among all Lipschitz domains with fixed volume.
\end{thm}

Inspired by the result of Theorem \ref{pol} it is natural to expect radially symmetrical configurations to be optimizers of some kind for $E$ (at least in the local sense).
From now on we will consider $\omega:= B_R$ (the open ball centered at the origin, whose radius, $0<R<1$, is chosen to verify the volume constraint $|B_R|=m$) and use shape derivatives to analyze this configuration. 
This technique has already been used by Conca and Mahadevan in \cite{conca} and Dambrine and Kateb in \cite{sensitivity} for the minimization of the first Dirichlet eigenvalue in a similar two-phase setting ($\Omega$ being a ball) and it can be applied with ease to our case as well. 

A direct calculation shows that the function $u$, solution to \eqref{pb2} where $\omega=\br$, has the following expression:
\begin{equation}\label{uexpl}
u(x)= \begin{cases}
\frac{1-R^2}{2N\sg_+}+\frac{R^2-|x|^2}{2N\sg_-} \quad &\text{for }|x|\in [0,R],\\[1ex]
\frac{1-|x|^2}{2N\sg_+} \quad&\text{for } |x|\in [R,1].
\end{cases}
\end{equation}
In this paper we will use the following notation for Jacobian and Hessian matrix respectively.
$$
(D{\bf v})_{ij}:=\frac{\partial v_i}{\partial x_j}, \quad (D^2 f)_{ij}= \frac{\partial^2 f}{\partial x_i\partial x_j},
$$
for all smooth real valued function $f$ and vector field ${\bf v}=(v_1,\dots, v_N)$ defined on $\Omega$.
We will introduce some differential operators from tangential calculus that will be used in the sequel. For smooth $f$ and $\bf v$ defined on $\partial \omega$ we set
\begin{equation}\label{diffop}
\begin{aligned}
\gr_\tau f &:= \gr \widetilde{f}-(\gr \widetilde{f}\cdot \n)\n \quad &\text{( tangential gradient)},   \\
\dv_\tau {\bf v}&:= \dv \widetilde{{\bf v}}-\n\cdot \left( D \widetilde{{\bf v}}\n\right) &\text{ (tangential divergence)},
\end{aligned}
\end{equation}
where $\widetilde{f}$ and $\widetilde{\bf v}$ are some smooth extensions on the whole $\Omega$ of $f$ and $\bf v$ respectively. It is known that the differential operators defined in \eqref{diffop} do not depend on the choice of the extensions.   
Moreover we denote by $D_\tau {\bf v}$ the matrix whose $i$-th row is given by $\gr_\tau v_i$.
We define the (additive) mean curvature of $\partial \omega$ as $H:=\dv_\tau \n$ (cf. \cite{henrot, SG}). According to this definition, the mean curvature $H$ of $\partial \br$ is given by  $(N-1)/R$.  


A first key result of this paper is the following. 
\begin{thm}\label{UNO}
For all suitable perturbations that fix the volume (at least at first order), the first shape derivative of $E$ at $B_R$ vanishes. 
\end{thm}
An improvement of Theorem \ref{UNO} is given by the following precise result (obtained by studying second order shape derivatives).
\begin{thm}\label{DUE}
Let $\sg_-,\sg_+>0$ and $R\in(0,1)$. If $\sg_->\sg_+$ then $B_R$ is a local maximizer for the functional $E$ under the fixed volume constraint.

On the other hand, if $\sg_-<\sg_+$ then $B_R$ is a saddle shape for the functional $E$ under the fixed volume constraint. 
\end{thm}

In section 2 we will give the precise definition of shape derivatives and introduce the famous Hadamard forumlas, a precious tool for computing them. In the end of section 2 a proof of Theorem \ref{UNO} will emerge as a natural consequence of our calculations. 
Section 3 will be devoted to the computation of the second order shape derivative of the functional $E$ in the case $\omega=\br$.
In Section 4 we will finally calculate the sign of the second order shape derivative of $E$ by means of the spherical harmonics. The last section contains an analysis of the different behaviour that arises when volume preserving transformations are replaced by surface area preserving ones.  

\section[Computation of the first order shape derivative:\\ Proof of Theorem \ref{UNO} ]{Computation of the first order shape derivative:\\ \large{Proof of Theorem \ref{UNO}}}

We consider the following class of perturbations with support compactly contained in $\Omega$:
$$
\mathcal{A}:=\Big\{ \Phi \in \C^\infty\big([0,1)\times \rn,\rn\big)  \;\Big|\;  \Phi(0,\cdot)= {\rm Id},\; \exists K\subset\subset \Omega \text{ s.t.}\; 
\Phi(t,x)=x \; \forall t\in [0,1) ,\; \forall x\in \rn\setminus K        \Big\}.
$$

For $\Phi \in \mathcal{A}$ we will frequently write $\Phi(t)$ to denote $\Phi(t,\cdot)$ and, for all domain $D$ in $\rn$, we will denote by $\Phi(t)(D)$ the set of all $\Phi(t,x)$ for $x\in D$. We will also use the notation $D_t:=\Phi(t)(D)$ when it does not create confusion.    
In the sequel the following notation for the first order approximation (in the ``time'' variable) of $\Phi$ will be used.
\begin{equation}\label{whosh}
\Phi(t)={\rm Id} + t \h +o(t) \quad \text{as }t\to0,
\end{equation}
where $\h$ is a smooth vector field. In particular we will write $h_n:=\h\cdot \n$ (the normal component of $\h$) and $\h_\tau : = \h- \hn \n$ on the interface. 
We are ready to introduce the definition of shape derivative of a shape functional $J$ with respect to a deformation field $\Phi$ in $\mathcal{A}$ as the following derivative along the path associated to $\Phi$.
$$
\restr{\frac{d}{dt} J(D_t)}{t=0} = \lim_{t\to 0} \frac{J(D_t)-J(D)}{t}.
$$

This subject is very deep. Many different formulations of shape derivatives associated to various  kinds of deformation fields have been proposed over the years. We refer to \cite{SG} for a detailed analysis on the equivalence between the various methods. For the study of second (or even higher) order shape derivatives and their computation we refer to \cite{SG, structure, simon, new}.   

The structure theorem for first and second order shape derivatives (cf. \cite{henrot}, Theorem 5.9.2, page 220 and the subsequent corollaries) yields the following expansion. 
For every shape functional $J$, domain $D$ and pertubation field $\Phi$ in $\mathcal{A}$, under suitable smoothness assumptions the following holds.

\begin{equation}\label{expansion}
J(D_t)=J(D)+t \, l_1^J (D)(h_n)+\frac{t^2}{2}\left(  l_2^J(D)(\hn,\hn)+l_1^J(D)(Z)  	\right) + o(t^2) \quad \text{ as }t\to 0,
\end{equation}
for some linear $l_1^J(D): \C^\infty (\partial D)\to \R$ and bilinear form $l_2^J(D): \C^\infty (\partial D)\times \C^\infty (\partial D)\to\R$ to be determined eventually. 
Moreover for the ease of notation we have set 
$$
Z:=\left( V'+D\h \h   \right)\cdot \n+ ((D_\tau \n) \htau)\cdot \htau-2\gr_\tau \hn\cdot\htau,
$$
where $V(t,\Phi(t)):=\partial_t \Phi(t)$ and $V':=\partial_t V(t,\cdot)$.

According to the expansion \eqref{expansion}, the first order shape derivative of a shape functional depends only on its first order apporximation by means of its normal components.
On the other hand the second order derivative contains an ``acceleration'' term $l_1^J(D)(Z)$.  
It is woth noticing that, (see Corollary 5.9.4, page 221 of \cite{henrot}) $Z$ vanishes in the special case when $\Phi=\id+t\h$ with $\h_\tau=\0$ on $\partial D$ (this will be a key observation to compute the bilinear form $l_2^J$).

We will state the following lemma, which will aid us in the computations of the linear and bilinear forms $l_1^J(D)$ and $l_2^J(D)$ for various shape functionals (cf.  \cite{henrot}, formula (5.17), page 176 and formulas (5.110) and (5.111), page 227).
\begin{lem}[Hadamrd's Formulas]\label{nonloso}
Take $\Phi\in \A$ and let $f=f(t,x)\in \C^2([0,T), L^1(\rn))\cap \C^1([0,T), W^{1,1}(\rn))$. For every smooth domain $D$ in $\rn$ define
$
J(D_t):=\int_{D_t} f(t)
$
(where we omit the space variable for the sake of readability). Then the following identities hold:
\begin{equation}\label{deri1}   
l_1^J(D)(\hn)= \int_D \restr{\partial_t f}{t=0} + \int_{\partial D} f(0) h_n,
\end{equation}
\begin{equation}\label{deri2}
l_2^J(D)(\hn,\hn)=\int_D \restr{\partial_{tt}^2 f}{t=0}+\int_{\partial D}2\restr{\partial_t f}{t=0}h_n+\big(H f(0)+\partial_n f(0) \big)h_n^2.
\end{equation}
\end{lem}

Since we are going to compute second order shape derivatives of a shape functional subject to a volume constraint, we will need to restric our attention to the class of perturbations in $\A$ that fix the volume of $\omega$:
$$
\mathcal{B}(\omega):= \big\{  \Phi\in \mathcal{A}  \;\big |\;  \vol\big(\Phi(t)(\omega)\big)=\vol(\omega)=m  \text{ for all } t\in [0,1)  \big\}.
$$
We will simply write $\B$ in place of $\B(\br)$.
Employing the use of Lemma \ref{nonloso} for the volume functional $\vol$ and of the expansion \eqref{expansion}, for all $\Phi\in\A$ we get
\begin{equation}\label{volex}
\vol(\omega_t)=\vol(\omega)+ t \int_{\partial\omega} \hn + \frac{t^2}{2} \left( \int_{\partial \omega} H h_n^2 + \int_{\partial\omega} Z     \right) + o(t^2) \text{ as }t\to0.
\end{equation}
This yields the following two conditions: 
\begin{align}
&\int_{\partial \omega} \hn =0,  \quad\quad\quad&\text{($1^{\rm st}$ order volume preserving)} \label{1st}\\
&\int_{\partial \omega} H h_n^2+ \int_{\partial \omega}Z=0. &\text{($2^{\rm nd}$ order volume preserving)}\label{2nd}
\end{align} 
\begin{rmk}
For every admissible perturbation field $\Phi=\id+t\h$ in $\A$, with $\h$ satisfying \eqref{1st}, we can find some perturbation field $\widehat{\Phi}\in\B$ such that $\widehat{\Phi}=\id + t \h + o(t)$ as $t\to 0$.
For example, the following construction works just fine:
$$
\widehat{\Phi}(t,x)=\frac{\Phi(t,x)}{\eta(x)\left( \frac{\vol(\Phi(t)(\omega))}{\vol(\omega)} \right)^{1/N}+(1-\eta(x))},
$$
where $\eta$ is a suitable smooth cutoff function compactly supported in $\Omega$ that attains the value $1$ on a neighbourhood of $\omega$. 
\end{rmk}

We will now introduce the concepts of ``shape" and ``material" derivative of a path of real valued functions defined on $\Omega$.
Fix an admissible perturbation field $\Phi\in\A$ and let $u=u(t,x)$ be defined on $[0,1)\times\Omega$ for some positive $T$.
Computing the partial derivative with respect to $t$ at a fixed point $x\in\Omega$ is usually called {\em shape derivative} of $u$; we will write:
$$
u'(t_0,x):= \frac{\partial u}{\partial t} (t_0,x), \;\text{ for }x\in\Omega, t_0\in [0,1). 
$$ 
On the other hand differentiating along the trajectories gives rise to the {\em material derivative}:
$$
\dot{u}(t_0,x):= \frac{\partial v}{\partial t}(t_0,x), \;x\in\Omega, t_0\in [0,1);
$$ 
where $v(t,x):=u(t, \Phi(t,x))$.
From now on for the sake of brevity we will omit the dependency on the ``time" variable unless strictly necessary and write $u(x)$, $u'(x)$ and $\dot{u}(x)$ for$u(0,x)$, $u'(0,x)$ and $\dot{u}(0,x)$. 
The following relationship between shape and material derivatives hold true:
\begin{align}
\label{u'a} u'&=\dot{u}-\gr u \cdot \h.
\end{align}
We are interested in the case where $u(t,\cdot):=u_{(B_R)_t}$ i.e. it is the solution to problem \eqref{pb} when $\omega=\Phi(t)(B_R)$.
In this case, since by symmetry we have $\gr u = (\dn u) \n$, the formula above admits the following simpler form on the interface $\partial B_R$:
\begin{equation}\label{u'}
u'=\dot{u}-(\dn u) \hn.
\end{equation}

It is natural to ask whether the shape derivatives of the functional $E$ are well defined. Actually, by a standard argument using the implicit function theorem for Banach spaces (we refer to \cite{conca, sensitivity} for the details) it can be proven that the application mapping every smooth vector field ${\bf h}$ compactly supported in $\Omega$ to $E\left((\id+\h)(\omega)\right)$ is of class $\mathcal{C}^\infty$ in a neighbourhood of ${\bf h }={\bf 0}$. This implies the shape differentiability of the functional $E$ for any admissible deformation field $\Phi\in \A$. As a byproduct we obtain the smoothness of the material derivative $\dot{u}$.

As already remarked in \cite{sensitivity} (Remark 2.1), in contrast to material derivatives, the shape derivative $u'$ of the solution to our problem has a jump through the interface. This is due to the presence of the gradient term in formula \eqref{u'a} (recall that the transmission condition provides only the continuity of the flux). On the other hand we will still be using shape derivatives because they are easier to handle in computations (and writing Hadamard formulas using them is simpler). 

\begin{prop}
For any given admissible $\Phi\in\A$, the corresponding $u'$ can be characterized as the (unique) solution to the following problem in the class of functions whose restriction to both $\br$ and $\Omega\setminus \overline{\br}$ is smooth:
\begin{equation}\label{u'pb}
\begin{cases}
\lap u' =0 \quad&\text{in } \br \cup (\Omega\setminus \overline{\br}),\\
\jump{\sg \dn u'}=0 \quad&\text{on }\bbr,\\
\jump{u'}=-\jump{\dn u}\hn\quad&\text{on } \bbr,\\
u'= 0 \quad&\text{on }\partial \Omega.
\end{cases}
\end{equation}
\end{prop}
\begin{proof}
Let us now prove that $u'$ solves \eqref{u'pb}.
First we take the shape derivative of both sides of the first equation in \eqref{pb2} at points away from the interface:
\begin{equation}\label{laplu'}
\Delta u' = 0 \text{ in }\omega \cup (\Omega \setminus \overline{\br}).
\end{equation}
In order to prove that $\jump{\sg\dn u'}$ vanishes on $\partial B_R$ we will proceed as follows.
We performing the change of variables $y:=\Phi(t,x)$ in \eqref{weak} and set $\varphi(x)=:\psi\left( \Phi(t,x) \right)$. Taking the derivative with respect to $t$, bearing in mind the first order approximation of $\Phi$ given by \eqref{whosh} yields the following.
\begin{align*}
\iO \sigma \left( - D\h \nabla u + \nabla \dot{u} \right)\cdot \nabla \psi - \iO \sg\gr u \cdot D\h \gr \psi + \iO \sg \gr u \cdot \gr \psi \dv{ \h} = \iO \psi \dv\h. 
\end{align*}

Rearranging the terms yields:
$$
\iO \sg \gr\dot{u}\cdot \gr\psi +\iO \sg \underbrace{\left( -D \h - D\h ^T + (\dv \h) I \right)\gr u \cdot \gr\psi}_{\mathlarger{\mathlarger{\mathlarger{\circledast}}}} = 
-\iO \h \cdot \gr\psi.
$$

Let {\bf x} and {\bf y} be two sufficiently smooth vector fields in $\rn$ such that $D {\bf x}= \left( D {\bf x}\right) ^T$ and $D {\bf y}= \left( D {\bf y}\right) ^T$. It is easy to check that the following identity holds:
$$
\left( -D \h - D\h^T + (\dv \h) I \right) {\bf x}\cdot {\bf y} = \dv \left(({\bf x}\cdot {\bf y}) \h\right) - \gr(\h\cdot {\bf x})\cdot {\bf y}- \gr(\h\cdot {\bf y})\cdot {\bf x}.
$$

We can apply this identity with ${\bf x}=\gr u$ and ${\bf y}= \gr \psi$ to rewrite $\mathlarger{\mathlarger{\mathlarger{\circledast}}}$ as follows:
$$
\left( -D \h - D\h ^T + (\dv \h) I \right)\gr u \cdot \gr\psi=\underbrace{\dv(\gr u \cdot \gr\psi \h)}_{\circled{1}} -\underbrace{\gr(\h\cdot \gr u)\cdot \gr \psi}_{\circled{2}} -\underbrace{\gr(\h\cdot \gr\psi)\cdot \gr u}_{\circled{3}}. 
$$
Thus 
\begin{align*}
\iO \sg \gr\dot{u}\cdot \gr \psi - \isl \jump{\sg \gr u \cdot \gr \psi}\hn- \iO \sg \gr(\h \cdot \gr u)\cdot \gr \psi+\isl \sg_- \dn u_- \jump{\dn \psi}\hn =0,
\end{align*}
where we have split the integrals and integrated by parts to handle the terms coming from $\circled{1}$ and $\circled{3}$. 

Now, merging together the integrals on $\Omega$ in the left hand side by \eqref{u'} and exploiting the fact that $\gr u = (\dn u) \n$ on $\bbr$, the above simplifies to 
\begin{equation}\label{this}
\iO \sg \gr u' \cdot \gr \psi = 0.
\end{equation}

Splitting the domain of integration and integrating by parts, we obtain
\begin{equation}\nonumber
\begin{aligned}
0=-\int_\br \sg_- \Delta u_-' \psi + \int_{\bbr} \sg_- \dn u_-' \psi -\int_{\Omega\setminus \overline{\br}} \sg_+ \Delta u_+' \psi -\int_{\bbr} \sg_+\dn u_+' \psi-\int_{\partial \Omega} \sg_+ \dn u_+' \psi \\
= \isl \jump{\sg \dn u'} \psi,
\end{aligned}
\end{equation}
where in the last equality we have used \eqref{laplu'} and the fact that $\psi$ vanishes on $\partial\Omega$.
By the arbitrariness of $\psi\in H_0^1(\Omega)$ we can conclude that $\jump{\sg \dn u'}=0$ on $\partial \br$.  
The remaining conditions of problem \eqref{u'pb} are a consequence of \eqref{u'a}.

To prove uniqueness for this problem in the class of functions whose restriction to both $\br$ and $\Omega\setminus \overline{\br}$ is smooth, just consider the difference between two solutions of such problem and call it $w$. Then $w$ solves
$$
\begin{cases}
\lap w =0 \quad&\text{in } \br \cup (\Omega\setminus \overline{\br}),\\
\jump{\sg \dn w}=0 \quad&\text{on }\bbr,\\
\jump{w}=0\quad&\text{on } \bbr,\\
w=0 \quad&\text{on }\partial \Omega;
\end{cases}
$$ 
in other words, $w$ solves 
$$
\begin{cases}
-\dv (\sg \gr w)=0 \quad&\text{in }\Omega,\\
w=0 \quad&\text{on }\partial\Omega.
\end{cases}
$$
Since the only solution to the problem above is the constant function $0$, uniqueness for Problem \eqref{u'pb} is proven.
\end{proof}

We emphasize that formulas \eqref{deri1} and \eqref{deri2} are valid only for $f$ belonging at least to the class $\C^2([0,T), L^1(\rn))\cap \C^1([0,T), W^{1,1}(\rn))$. 
We would like to apply them to $f(t)=u_t$ and $f(t)=\sigma_t |\nabla u_t|^2$, where $\sg_t$ and $u_t$ are the distribution of conductivities and the solution of problem \eqref{pb} respectively corresponding to the case $\omega=(\br)_t$. On the other hand, $u_t$ is not regular enough in the entire domain $\Omega$, despite being fairly smooth in both $\omega_t$ and $\Omega_t\setminus \overbar{\omega_t}$: therefore we need to split the integrals in order to apply \eqref{deri1} and \eqref{deri2} (this will give rise to interface integral terms by integration by parts).

\begin{thm}\label{thm1}
For all $\Phi\in\A$ we have
$$l_1^E(B_R)(\hn)=-\isl \jump{\sg |\gr u|^2} \hn.$$
In particular, for all $\Phi$ satisfying the first order volume preserving condition \eqref{1st} we get $l_1^E(B_R)(\hn)=0$.
\end{thm}
\begin{proof}
We apply formula \eqref{deri1} to 
$
E(\omega_t)=\int_\Omega u_t=\int_{\omega_t} u_t+\int_{\compt} u_t
$
to get
$$
l_1^E(B_R)(\hn)= \int_\br u_-'+\int_{\bbr}u_- h_n + \int_{\Omega\setminus \overline{B_R}} u_+'
-\int_{\bbr} u_+ h_n.
$$
Using the jump notation we rewrite the previous expression as follows
\begin{equation}\label{sonoichi}
l_1^E(B_R)(\hn)= \int_\Omega u' - \int_{\bbr} [u h_n]=\int_\Omega u';
\end{equation}
notice that the surface integral in \eqref{sonoichi} vanishes as both $u$ and $h_n$ are continuous through the interface.
\\
Next we apply \eqref{deri1} to $E(\omega_t)=\int_\Omega \sigma_t |\nabla u_t|^2$.
\begin{align*}
l_1^E(B_R)(\hn)= & 2\int_{\br} \sigma_- \nabla u_-\cdot \nabla u_-' + \int_{\bbr} \sigma_- |\nabla u_-|^2 h_n+ \\
& 2\int_{\compl} \sigma_+ \nabla u_+\cdot \nabla u_+' +\int_{\partial\Omega} \sigma_+ |\nabla u_+|^2 h_n - \int_{\bbr} \sigma_+ |\nabla u_+|^2 h_n.
\end{align*}
Thus we get the following:
\begin{equation}\label{sononi}
l_1^E(B_R)(\hn)= 2\int_\Omega \sigma \nabla u\cdot \nabla u' -\isl \jump{\sigma |\nabla u|^2}h_n.
\end{equation}
Comparing \eqref{this} (choose $\psi=u$) with 
\eqref{sononi} gives
\begin{equation}\label{that}
l_1^E(B_R)(\hn)= -\isl \jump{\sg |\gr u|^2} \hn.
\end{equation}
By symmetry, the term $ \jump{\sigma |\nabla u|^2}$ is constant on $\bbr$ and can be moved outside the integral sign. Therefore we have
\[
l_1^E(B_R)(\hn)= 0 \text{ for all $\Phi$ satisfying \eqref{1st}}.
\]
This holds in particular for all $\Phi\in\B$. 
\end{proof}

\section{Computation of the second order shape derivative}

The result of the previous chapter tells us that the configuration corresponding to $\br$ is a critical shape for the functional $E$ under the fixed volume constraint. In order to obtain more precise information, we will need an explicit formula for the second order shape derivative of $E$.
The main result of this chapter consists of the computation of the bilinear form $l_2^E(B_R)(\hn,\hn)$.

\begin{thm}\label{lduu}
For all $\Phi\in\A$ we have 
$$
\ldue=-2\isl \sg_-\dn u_- \jump{\dn u'}\hn -2\isl \sg_-\dn u_- \jump{\partial_{nn}^2 u}h_n^2-\isl\sg_- \dn u_-\jump{\dn u} H h_n^2.
$$
\end{thm}
\begin{proof}
Take $\Phi=\id+t\h$ in $\A$ with $\h=\hn \n$ on $\bbr$. As remarked after \eqref{expansion}, $Z$ vanishes in this case. We get 
$$
\ldue=\frac{d^2}{dt^2}\restr{E(\Phi(t)(B_R))}{t=0}.
$$
Hence, substituting the expression of the first order shape derivative obtained in Theorem \ref{thm1} yields
\begin{equation}\nonumber
\ldue= -\restr{\frac{d}{dt}\left( \int_{(\partial\br)_t} \jump{\sg_t |\gr u_t|^2}\hn\circ \Phi^{-1}(t)   \right) }{t=0}.
\end{equation}
We unfold the jump in the surface integral above and apply the divergence theorem to obtain
\begin{equation}\label{ltuu}
\ldue=\frac{d}{dt}\restr{\left( \int_{(B_R)_t} \dv\left(  \sg_- |\gr u_t|^2 \h\circ \Phi^{-1}(t) \right)  -  \int_{\Omega \setminus (B_R)_t} \dv\left(  \sg_+ |\gr u_t|^2 \h\circ \Phi^{-1}(t) \right)  \right)}{t=0}
\end{equation}
We will treat each integral individually.
By \eqref{whosh} we have $\partial_t \restr{\left(  \Phi  \right)}{t=0}=-\h$, therefore $\partial_t\restr{\left(\h\circ \Phi^{-1} \right)}{t=0}$=-D\h\h. Now set $f(t):=\sg_- |\gr u_t|^2$.
By \eqref{deri1} we have
\begin{equation}\nonumber
\begin{aligned}
\frac{d}{dt}\restr{\left(\int_{(B_R)_t}   \dv \left( f(t) \h\circ \Phi^{-1}(t) \right)\right)}{t=0} =   \underbrace{ \int_\br \partial_t\restr{\left(  \dv \left(  f(t) \h\circ\Phi^{-1}(t)  \right)  \right)}{t=0}}_{(A)}
 + \underbrace{ \isl \dv (f(0)\h)\hn}_{(B)}.
\end{aligned}
\end{equation}
We have
\begin{equation}\nonumber
\begin{aligned}
(A)=\int_\br \dv \left( \partial_t \restr{f}{t=0} \h + f(0)\partial_t \restr{(\h\circ \Phi^{-1}(t))}{t=0}  \right) =    
  \int_\br \dv \left( \partial_t \restr{f}{t=0} \h \right) - \int_\br \dv \left( f(0)D\h\h\right) =  \\
    \int_\br \partial_t \restr{f}{t=0} \hn - \isl f(0) \n\cdot D\h\h.  
\end{aligned}
\end{equation}
On the other hand
$$
(B)= \isl \dv \left(  f(0) \h \right)\hn = \isl \left(  \gr f(0)\cdot \h + f(0)\dv \h   \right) \hn.
$$
Using the fact that $\h=\hn\n$ and $\dv\h-\n\cdot D\h\n=: \dv_\tau (\hn\n)=H\hn$ (c.f. Equation (5.22), page 366 of \cite{SG}) we get 
$$
(A)+(B)= \isl f' \hn + \isl (\dn f+ Hf)h_n^2.
$$
Substituting $f(t)=\sg_-|\gr u_t|^2$ yields
$$
(A)+(B)= 2 \isl \sg_- \gr u_- \cdot \gr u_-' \hn + 2\isl \sg_- \dn u_- \partial_{nn}^2 u_- h_n^2 + \isl \sg_- |\gr u_-|^2 H h_n^2.
$$
The calculation for the integral over $\Omega\setminus (\br)_t$ in \eqref{ltuu} is analogous.
We conclude that 
$$
\ldue= -2 \isl \sg_- \dn u_- \jump{\dn u'}\hn -2\isl \sg_- \dn u_- \jump{\partial_{nn}^2 u}h_n^2- \isl\sg_- \dn u_-\jump{\dn u} H h_n^2.
$$
\end{proof}

\section[Classification of the critical shape $\br$:\\ Proof of Theorem \ref{DUE} ]{Classification of the critical shape $\br$:\\ \large{Proof of Theorem \ref{DUE}}}

In order to classify the critical shape $\br$ of the functional $E$ under the volume constraint we will use the expansion shown in \eqref{expansion}.
For all $\Phi\in\B$ and $t>0$ small, it reads
\begin{equation}\label{eex}
E\left( \Phi(t)(\br)\right)= E(\br)+ \frac{t^2}{2}\left(   \ldue-\isl \jump{\sg |\gr u|^2} Z  \right)+o(t^2).
\end{equation}
Employing the use of the second order volume preserving condition \eqref{2nd} and the fact that, by symmetry, the quantity $\jump{\sg |\gr u|^2}$ is constant on the interface $\bbr$ we have
$$
-\isl \jump{\sg |\gr u|^2} Z = \isl \jump{\sg |\gr u|^2} H h_n^2.
$$
Combining this with the result of Theorem \ref{lduu} yields
$$
E\left( \Phi(t)(\br)\right)= E(\br)+ {t^2}{\left\{  -\isl \sg_-\dn u_-\jump{\dn u'}\hn -\isl \sg_-\dn u_- \jump{\partial_{nn}^2 u}h_n^2         \right\}}+o(t^2).
$$
We will denote the expression between braces in the above by $Q(\hn)$. Since $u'$ depends linearly on $\hn$ (see \eqref{u'pb}), it follows immediately that $Q(\hn)$ is a quadratic form in $\hn$.
Since both $u$ and $u'$ verify the transmission condition (see \eqref{pb2} and \eqref{u'pb}) we have 
$$
\sg_-\dn u_-\jump{\dn u'}=\jump{\sg\dn u\dn u'}= \sg_-\dn {u'}_-\jump{\dn u} \text{ on }\bbr.
$$
Using the explicit expression of $u$ given in \eqref{uexpl}, after some elementary calculation we write
$$
Q(\hn)=\frac{R}{N}\left(\frac{1}{\sg_-}-\frac{1}{\sg_+}\right)\left( -\isl \sg_- \dn {u'}_- \hn +\frac{1}{N}\isl h_n^2  \right).
$$ 

In the following we will try to find an explicit expression for $u'$. To this end we will perform the spherical harmonic expansion of the function $\hn:\bbr \to \R$.
We set 
\begin{equation}\label{hexp}
\hn(R\theta)=\sum_{k=1}^\infty \sum_{i=1}^{d_k} \alpha_{k,i} Y_{k,i}(\theta) \quad \text{ for all } \theta\in \partial B_1 
.
\end{equation}
The functions $Y_{k,i}$ are called {\em spherical harmonics} in the literature. They form a complete orthonormal system of $L^2(\partial B_1)$ and are defined as the solutions of the following eigenvalue problem:
$$
-\lap_\tau Y_{k,i}=\lambda_k Y_{k,i} \quad \text{ on }\partial B_1,\\
$$
where $\lap_\tau:= \dv_\tau \gr_\tau$ is the Laplace-Beltrami operator on the unit sphere.
We impose the following normalization coniditon 
\begin{equation}\label{normalization}
\int_{\partial B_1} Y_{k,i}^2=R^{1-N}.
\end{equation}
The following expressions for the eigenvalues $\lambda_k$ and the corresponding multiplicities $d_k$ are also known:
\begin{equation}\label{lambdak}
\lambda_k= k(k+N-2), \quad \quad d_k= \binom{N+k-1}{k}-\binom{N+k-2}{k-1}.
\end{equation}
Notice that the value $k=0$ had to be excluded from the summation in \eqref{hexp} because $\hn$ verifies the first order volume preserving condition \eqref{1st}.

Let us pick an arbitrary $k\in\{1,2,\dots\}$ and $i\in\{1,\dots, d_k\}$. We will use the method of separation of variables to find the solution of problem \eqref{u'pb} in the particular case when $\hn(R\theta)=Y_{k,i}(\theta)$, for all $\theta\in \partial B_1$.
\begin{figure}[h]
\centering
\includegraphics[width=0.8\textwidth]{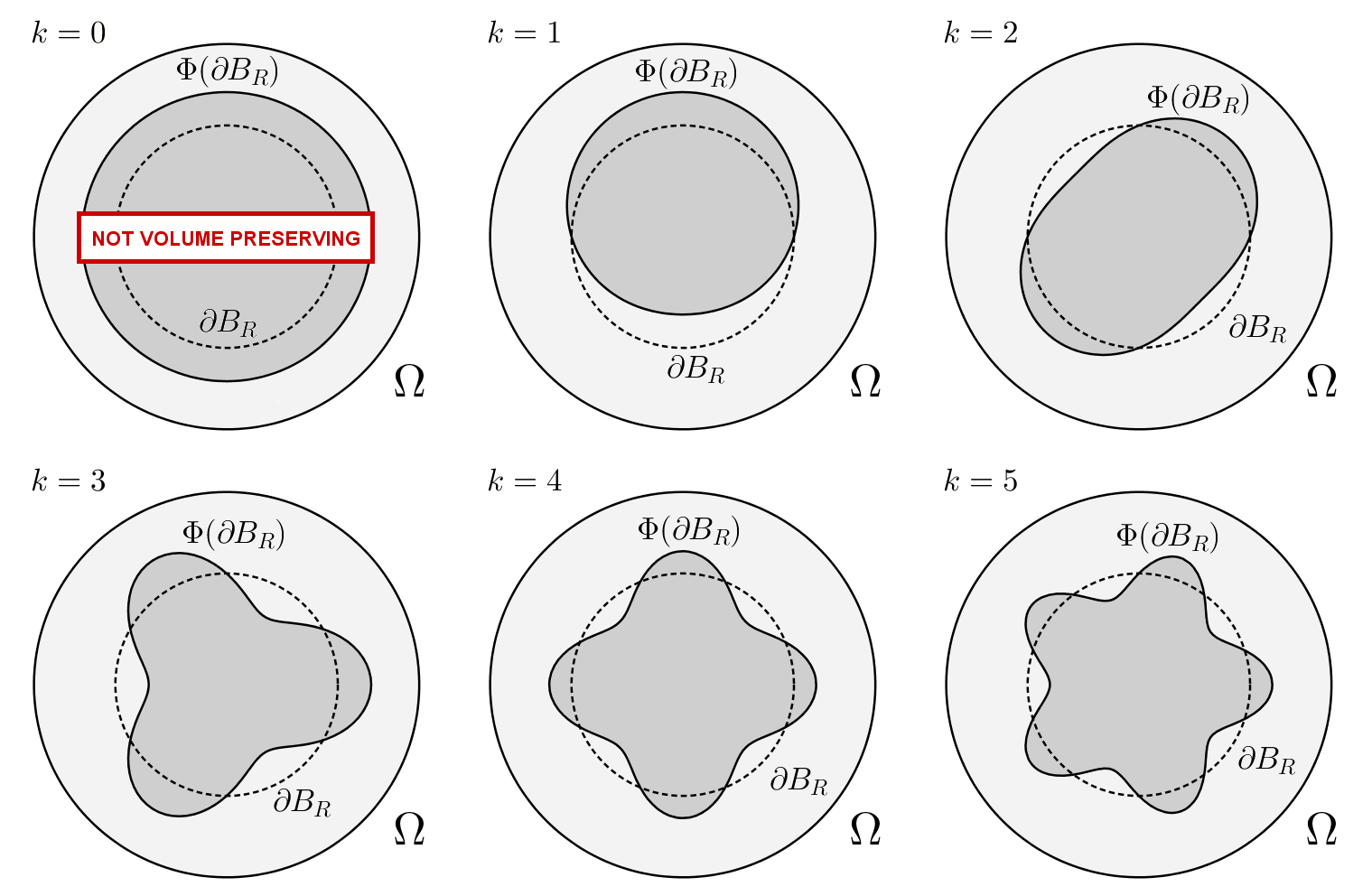}
\caption{How $\Phi(t)(B_R)$ looks like for small $t$ when $\hn(R\cdot)=Y_{k,i}$, in 2 dimensions.}
\label{sphar}
\end{figure}

Set $r:=|x|$ and, for $x\ne 0$, $\theta:=x/|x|$. 
We will be searching for solutions to \eqref{u'pb} of the form $u'=u'(r,\theta)=f(r)g(\theta)$.
Using the well known decomposition formula for the Laplacian into its radial and angular components, the equation $\lap u'=0$ in $\br \cup (\Omega\setminus \overline{\br})$ can be rewritten as
$$
0=\lap u'(x) = f_{rr}(r)g(\theta)+\frac{N-1}{r}f_r(r)g(\theta)+\frac{1}{r^2}f(r)\lap_{\tau}g(\theta) \;\text{for }r\in(0,R)\cup (R,1),\, \theta\in\partial B_1.
$$
Take $g=Y_{k,i}$.
Under this assumption, we get the following equation for $f$:
\begin{equation}\label{f}
f_{rr}+\frac{N-1}{r}f_r-\frac{\lambda_k}{r^2}f=0 \quad\text{in } (0,R)\cup (R,1).
\end{equation}
It can be easily checked that the solutions to the above consist of linear combinations of  $r^\eta$ and $r^\xi$, where
\begin{equation}\label{xieta}
\begin{aligned}
\eta&=\eta(k)=\frac{1}{2} \left( 2-N + \sqrt{ (N-2)^2+4\lambda_k } \right)=k,\\
\xi&=\xi(k)= \frac{1}{2} \left( 2-N - \sqrt{ (N-2)^2+4\lambda_k } \right)=2-N-k.
\end{aligned}
\end{equation}
Since equation \eqref{f} is defined for $r\in (0,R)\cup (R,1)$, we have that the following holds for some real constants $A$, $B$, $C$ and $D$;
$$
f(r)= \begin{cases}
Ar^{2-N-k}+Br^k \quad&\text{for } r\in(0,R),\\
Cr^{2-N-k}+Dr^k \quad&\text{for } r\in(R,1).
\end{cases}
$$
Moreover, since ${2-N-k}$ is negative, $A$ must vanish, otherwise a singularity would occur at $r=0$.
The other three constants can be obtained by the interface and boundary conditions of problem \eqref{u'pb} bearing in mind that $u'(r,\theta)=f(r)Y_{k,i}(\theta)=f(r)h_n(R \theta)$.
We get the following system:
$$
\begin{cases}
C R^{2-N-k}+ DR^k-BR^k= -\frac{R}{N\sg_-} + \frac{R}{N \sg_+},\\
\sg_- kB  R^{k-1}= \sg_+ {(2-N-k)} C R^{{2-N-k}} + \sg_+ k D R^{k-1},\\
C+D=0.
\end{cases}
$$
Although this system of equations could be easily solved completely for all its indeterminates, we will just need to find the explicit value of $B$ in order to go on with our computations.
We have 
\begin{equation}\label{B}
B=B_k=\frac{R^{1-k}}{N\sg_-} \cdot \frac{k(\sg_--\sg_+)R^k-(2-N-k)(\sg_--\sg_+)R^{2-N-k}}{k(\sg_--\sg_+)R^k+((2-N-k)\sg_+-k\sg_-)R^{2-N-k}}.
\end{equation}
Therefore, in the particular case when $\hn(R\,\cdot)=Y_{k,i}$, 
$$
{u'}_-={u'}_-(r,\theta)=B_k r^k Y_{k,i}(\theta), \quad r\in[0,R), \; \theta\in\partial B_1,
$$ 
where $B_k$ is defined as in \eqref{B}.
By linearity, we can recover the expansion of ${u'}_-$ in the general case (i.e. when \eqref{hexp} holds):
\begin{equation}
\begin{aligned}
&{u'}_-(r,\theta)= \sum_{k=1}^\infty \sum_{i=1}^{d_k} \alpha_k B_k r^k Y_{k,i}(\theta), \quad r\in[0,R),\; \theta\in\partial B_1, \quad\text{and therefore} \\
&\dn {u'}_-(R,\theta)=\sum_{k=1}^\infty \sum_{i=1}^{d_k} \alpha_k B_k k R^{k-1} Y_{k,i}(\theta), \quad\quad\quad\theta\in\partial B_1.
\end{aligned}
\end{equation}
We can now diagonalize the quadratic form $Q$, in other words we can consider only the case $\hn(R\,\cdot)=Y_{k,i}$ for all possible pairs $(k,i)$. 
We can write $Q$ as a function of $k$ as follows:
\begin{equation}\label{Q}
\begin{aligned}
Q(\hn)=Q(k)=\frac{R}{N}\left(\frac{\sg_+-\sg_-}{\sg_+\sg_-}\right)\left(-\sg_-B_k k R^{k-1} +\frac{1}{N}\right)=\\
\frac{R}{N^2}\left( \frac{\sg_+-\sg_-}{\sg_+\sg_-}\right)\left(1- k \frac{k(\sg_--\sg_+)R^k-(2-N-k)(\sg_--\sg_+)R^{2-N-k}}{k(\sg_--\sg_+)R^k+((2-N-k)\sg_+-k\sg_-)R^{2-N-k}}  \right).
\end{aligned}
\end{equation}

The following lemma will play a central role in determining the sign of $Q(k)$ and hence proving Theorem \ref{mainthm}.
\begin{lem}\label{lemmone}
For all $R\in (0,1)$ and $\sg_-,\sg_+>0$, the function $k\mapsto Q(k)$ defined in \eqref{Q} is monotone decreasing for $k\ge 1$. 
\end{lem}
\begin{proof}
Let us denote by $\rho$ the ratio of the the conductivities, namely $\rho:= \sg_-/\sg_+$.
We get 
$$
Q(k)=\frac{R}{N^2}\left( \frac{1-\rho}{\sg_-}\right)\left(1- k \frac{k(\rho-1)R^k-(2-N-k)(\rho-1)R^{2-N-k}}{k(\rho-1)R^k+((2-N-k)-k\rho)R^{2-N-k}}  \right).
$$
In order to prove that the map $k\mapsto Q(k)$ is monotone decreasing it will be sufficient to prove that the real function
$$
j(x):=x\frac{x-(2-N-x)R^{2-N-2x}}{(1-\rho)x+ \left(-2+N+x+\rho x \right)R^{2-N-2x}}
$$
is monotone increasing in the interval $(1,\infty)$. Notice that this does not depend on the sign of $\rho-1$. 
From now on we will adopt the following notation:
\begin{equation}\nonumber
L:=R^{-1}>1,\quad M:=N-2\ge 0,\quad P=P(x):= L^{2x+M}.
\end{equation}
Using the notation introduced above, $j$ can be rewritten as follows 
$$
j(x)=\frac{x^2+(x^2+Mx)P}{(1-\rho)x+(x+M+\rho x)P}.
$$
In order to prove the monotonicity of $j$, we will compute its first derivative and then study its sign.
We get 
$$
j'(x)=\frac{MP(MP+2Px+2x)+x^2(P+1)^2+\rho x^2P (P-1/P-4x\log(L)-2M\log(L))}{\left((1-\rho)x+(x+M+\rho x)P \right)^2}.
$$ 
The denominator in the above is positive and we claim that also the numerator is. To this end it suffices to show that the quantity multiplied by $\rho x^2P$ in the numerator, namely $P-1/P-4x\log(L)-2M\log(L)$, is positive for $x\in (1,\infty)$ (although, we will show a stronger fact, namely that it is positive for all $x>0$).
$$
\frac{d}{dx}\left(P-\frac{1}{P}-4x\log(L)-2M\log(L)\right)= 2{\log(L)}{\left(P+\frac{1}{P}-2\right)}>0 \quad\text{for }x>0,
$$
where we used the fact that $L>1$ and that $P\mapsto P+P^{-1}-2$ is a non-negative function vanishing only at $P=1$ (which does not happen for positive $x$).
We now claim that 
$$
\restr{\left(P-\frac{1}{P}-4x\log(L)-2M\log(L)\right)}{x=0}= L^M-\frac{1}{L^M}-2M \log(L)\ge 0.
$$
This can be proven by an analogous reasoning: treating $M$ as a real variable and differentiating with respect to it, we obtain
$$
\frac{d}{dM}\left( L^M-\frac{1}{L^M}-2M \log(L) \right) = \log(L)\left(L^M+\frac{1}{L^M}-2\right)\ge 0
$$
(notice that the equality holds only when $M=0$), moreover,
$$
\restr{\left( L^M-\frac{1}{L^M}-2M \log(L) \right)}{M=0}=0,
$$
which proves the claim.
\end{proof}
We are now ready to prove the main result of the paper.
\begin{thm}\label{mainthm}
Let $\sg_-,\sg_+>0$ and $R\in(0,1)$. If $\sg_->\sg_+$
then 
$$
\frac{d^2}{dt^2}\restr{E\big(\Phi(t)(B_R)\big)}{t=0}<0 \quad \text{ for all }\Phi\in\B.
$$

Hence, $B_R$ is a local maximizer for the functional $E$ under the fixed volume constraint. 
On the other hand, if $\sg_-<\sg_+$, then there exist some $\Phi_1$ and $\Phi_2$ in $\B$, such that 
$$
\frac{d^2}{dt^2}\restr{E\big(\Phi_1(t)(B_R)\big)}{t=0}<0,\quad \frac{d^2}{dt^2}\restr{E\big(\Phi_2(t)(B_R)\big)}{t=0}>0.
$$
In other words, $B_R$ is a saddle shape for the functional $E$ under the fixed volume constraint. 
\end{thm}
\begin{proof}
\begin{figure}[h]
\centering
\includegraphics[scale=0.5]{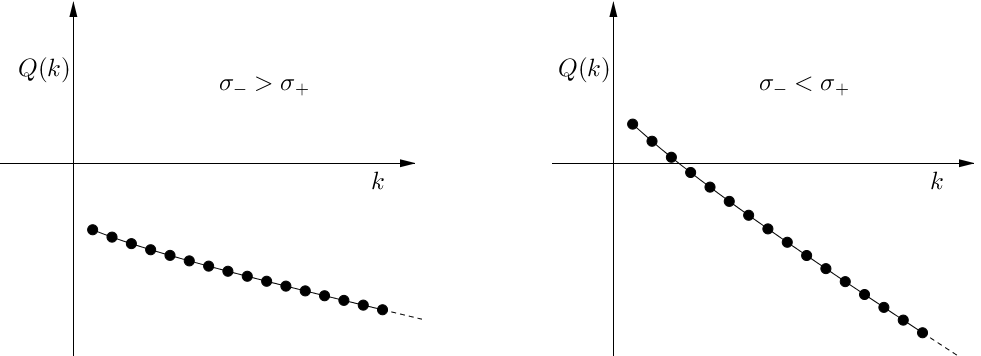}
\label{q}
\end{figure}
We have 
\begin{equation}\label{quuno}
Q(1)=\frac{R}{N^2}\left( \frac{1-\rho}{\sg_-}\right) \frac{N\rho}{\rho(R^N+1)     +N-R^N-1}.
\end{equation}
Since $N\ge 2$, $R\in (0,1)$, we have $N-R^N-1>0$ and therefore it is immediate to see that $Q(1)$ and $1-\rho$ have the same sign. 

If $\sg_->\sg_+$, then, by Lemma \ref{lemmone}, we get in particular that $Q(k)$ is negative for all values of $k\ge 1$. This implies that the second order shape derivative of $E$ at $B_R$ is negative for all $\Phi\in\B$ and therefore $B_R$ is a local maximizer for the functional $E$ under the fixed volume constraint as claimed.

On the other hand, if $\sg_-<\sg_+$, by \eqref{quuno} we have $Q(1)>0$. 
An elementary calculation shows that, for all $\sg_-,\sg_+>0$,    
$$
\lim_{k\to\infty} Q(k)=-\infty.
$$
Therefore, when $\sg_-<\sg_+$, $B_R$ is a saddle shape for the functional $E$ under the fixed volume constraint.
\end{proof}

\section{The surface area preserving case}
The method employed in this paper can be applied to other constraints without much effort. For instance, it might be interesting to see what happens when volume preserving perturbations are replaced by surface area preserving ones. Is $B_R$ a critical shape for the functional $E$ even in the class of domains of fixed surface area? If so, of what kind?
We set $\per(D):=\int_{\partial D} 1$ for all smooth bounded domain $D\subset \rn$. 
The following expansion for the functional $\per$ can be obtained just as we did for \eqref{volex}:
\begin{equation}\nonumber
\per(\omega_t)=\per(\omega)+ t \int_{\partial\omega}H \hn + \frac{t^2}{2} \left( l_2^\per(\omega) (\hn,\hn) + \int_{\partial\omega}H Z     \right) + o(t^2) \text{ as }t\to0,
\end{equation}
where, (cf. \cite{henrot}, page 225)  
\begin{equation}\label{l2per}
l_2^\per(\omega)(\hn,\hn)=\int_{\partial \omega} |\gr_\tau \hn|^2 +\int_{\partial \omega} \left( H^2 - {\rm tr}\big((D_\tau \n)^T D_\tau \n\big) \right) h_n^2.
\end{equation}
We get the following first and second order surface area preserving conditions. 
\begin{align}\label{1st2nd}
&\int_{\partial \omega}H \hn =0, 
&\int_{\partial \omega} |\gr_\tau \hn|^2 +\int_{\partial \omega} \left( H^2 - {\rm tr}\big((D_\tau \n)^T D_\tau \n\big) \right) h_n^2+ \int_{\partial \omega}H Z=0. 
\end{align}

Notice that when $\omega=B_R$, the first order surface area preserving condition is equivalent to the first order volume preserving condition \eqref{1st} and therefore, by Theorem \ref{thm1}, $B_R$ is a critical shape for $E$ under the fixed surface area constraint as well.   

The study of the second order shape derivative of $E$ under this constraint is done as follows. Employing the use of \eqref{eex} together with the second order surface area preserving condition in \eqref{1st2nd} we get
$$
\frac{d^2}{dt^2}\restr{E\big(\Phi(t)(B_R)\big)}{t=0}= l_2^E(B_R)(\hn,\hn) + \frac{\jump{\sg |\gr u|^2}}{H} l_2^\per(B_R)(\hn,\hn).
$$
In other words, we managed to write the shape Hessian of $E$ as a quadratic form in $\hn$. We can diagonalize it by considering $\hn(R\cdot)=Y_{k,i}$ for all possible pairs $(k,i)$, where we imposed again the normalization \eqref{normalization}. Under this assumption, by \eqref{lambdak} we get
$$
\int_\bbr |\gr_\tau \hn|^2= \frac{\lambda_k}{R^2}=\frac{k(k+N-2)}{R^2}.
$$  
We finally combine the expression for $l_2^E(B_R)$ of Theorem \ref{lduu} with that of $l_2^\per$ \eqref{l2per} to obtain
$$
E\big(\Phi(t)(B_R)\big)= E(B_R) + t^2 \,\widetilde{Q}(k) + o(t^2) \text{ as }t\to 0,
$$
where
$$
\widetilde{Q}(k)=\frac{R}{N^2}\left( \frac{1-\rho}{\sg_-}\right)\left(\frac{3}{2}     - \frac{k(k+N-2)}{2(N-1)} -k\,\frac{k(\rho-1)R^k-(2-N-k)(\rho-1)R^{2-N-k}}{k(\rho-1)R^k+((2-N-k)-k\rho)R^{2-N-k}} \right).
$$

It is immediate to check that $\widetilde{Q}(1)=Q(1)$ and therefore, $\widetilde{Q}(1)$ is negative for $\sg_->\sg_+$ and positive otherwise. On the other hand, $\lim_{k\to\infty} \widetilde{Q}(k)=\infty$ for $\sg_->\sg_+$ and $\lim_{k\to\infty} \widetilde{Q}(k)=-\infty$ for $\sg_-<\sg_+$. In other words, under the surface area preserving constraint $B_R$ is always a saddle shape, independently of the relation between $\sg_-$ and $\sg_+$. 
\begin{figure}[h]
\centering
\includegraphics[scale=0.5]{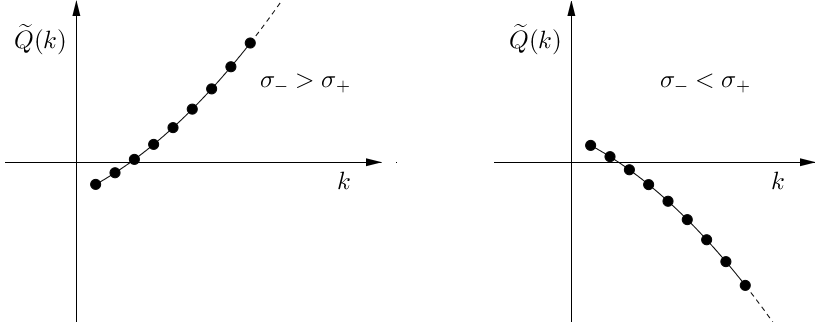}
\label{qtil}
\end{figure}

We can give the following geometric interpretation to this unexpected result. 
Since the case $k=1$ corresponds to deformations that coincide with translations at first order, it is natural to expect a similar behaviour under both volume and surface area preserving constraint. On the other hand, high frequency perturbations (i.e. those corresponding to a very large eigenvalue) lead to the formation of indentations in the surface of $B_R$. Hence, in order to prevent the surface area of $B_R$ from expanding, its volume must inevitably shrink (this is due to the higher order terms in the expansion of $\Phi$). This behaviour can be confirmed by looking at the second order expansion of the volume functional under the effect of a surface area preserving transformation $\Phi\in\A$ on the ball:
$$
\vol\big(\Phi(t)(\br)\big)=\vol(\br)+\frac{t^2}{2}\left( \frac{1}{R}-\frac{k(k+N-2)}{(N-1)R} \right) +o(t^2) \text{ as }t\to 0.
$$
We see that the second order term vanishes when $k=1$, while getting arbitrarily large for $k\gg1$.
Since this shrinking effect becomes stronger the larger $k$ is, this suggests that the behaviour of $E\big(\Phi(t)(B_R)\big)$ for large $k$ might be approximated by that of the extreme case $\omega=\emptyset$. For instance, when $\sg_->\sg_+$ we have that $E(B_R)<E(\emptyset)$ and this is coherent with what we found, namely $\widetilde{Q}(k)>0$ for $k\gg 1$.

\section*{Acknowledgments} 
This paper is prepared as a partial fulfillment of the author's doctor's degree at Tohoku University.
The author would like to thank Professor Shigeru Sakaguchi (Tohoku University) for his precious help in finding interesting problems and for sharing his naturally optimistic belief that they can be solved.
Moreover we would like to thank the anonymous referee, who suggested to study the surface area preserving case and helped us find a mistake in our calculations. Their detailed analysis and comments on the previous version of this paper, contributed to make the new version shorter and more readable.

\bigskip
\noindent
\textsc{
Research Center for Pure and Applied Mathematics, Graduate
School of
Information Sciences, Tohoku University, Sendai 980-8579
, Japan.
} \\
\noindent
{\em Electronic mail address:}
cava@ims.is.tohoku.ac.jp

\end{document}